\documentclass[final]{siamltex}
\usepackage{amsfonts}
\usepackage{amssymb,amsmath}
\usepackage{showkeys}
\usepackage{amssymb}

\newcommand{\ppsi}{\varphi}
\newcommand{\muy}{\mu^y}
\newcommand{\ud}{\mathrm{d}}
\newcommand{\e}{\mathrm{e}}

\newcommand{\R}{\mathbb{R}}

\newcommand{\Z}{\mathbb{Z}}
\newcommand{\TT}{\mathbb{T}}

\newcommand{\cA}{\mathcal{A}}

\newcommand{\cdiv}{\nabla\cdot}
\newcommand{\lt}{\left}
\newcommand{\rt}{\right}
\newcommand{\paro}{\partial D}

\newcommand{\kk}{\tiny{\rm K}}
\newcommand{\cN}{\mathcal{N}}
\newcommand{\bbR}{\mathbb{R}}
\newcommand{\bbT}{\mathbb{T}}
\newcommand{\cG}{\mathcal{G}}
\newcommand{\bbP}{\mathbb{P}}

\newcommand{\bbE}{\mathbb{E}}

\newcommand{\eps}{{\epsilon}}
\newcommand{\dhh}{d_{\mbox {\tiny{\rm Hell}}}}

\newcommand{\la}{\langle}
\newcommand{\ra}{\rangle}

\newtheorem{asp}[theorem]{Assumption}

\begin{document}

\title{Uncertainty quantification and weak approximation
of an elliptic inverse problem}

\author{M. Dashti,\ A. M. Stuart\thanks{Mathematics Institute, University of Warwick, Coventry CV4 7AL, UK}} 

\maketitle

\begin{abstract}

We consider the inverse problem of determining the
permeability from the pressure in a Darcy model of
flow in a porous medium. Mathematically the problem
is to find the diffusion coefficient for a linear 
uniformly elliptic partial differential equation
in divergence form, in a bounded domain in dimension 
$d \le 3$, from measurements of the solution
in the interior. 

We adopt a Bayesian approach to the problem. We place
a prior random field measure on the log permeability,
specified through the Karhunen-Lo\`eve expansion of its draws.
We consider Gaussian measures constructed this way,
and 
study the regularity of functions drawn from them.
We also study the Lipschitz properties of the observation
operator mapping the log permeability to the observations.
Combining these regularity 
and continuity estimates, we show that the posterior measure 
is well-defined on a suitable Banach space.  
Furthermore the posterior measure is 
shown to be Lipschitz with respect to the data 
in the Hellinger metric, giving rise 
to a form of well-posedness of the inverse problem.
Determining the posterior measure, given the data,
solves the problem of uncertainty quantification
for this inverse problem.

In practice the posterior measure must be approximated
in a finite dimensional space.
We quantify the errors incurred by employing
a truncated Karhunen-Lo\`eve expansion to represent
this meausure.  In particular we study weak convergence 
of a general class of locally Lipschitz functions
of the log permeability, and apply this general
theory to estimate errors in 
the posterior mean of the pressure and the
pressure covariance, under refinement 
of the finite dimensional Karhunen-Lo\`eve truncation.

\end{abstract}

\section{Introduction}\label{sec:intro}

There is a growing interest in uncertainty quantification for differential equations
in which the input data is uncertain. In the context of elliptic partial differential equations
much of this work has concentrated on the problem of groundwater flow in which
uncertainty enters the diffusion coefficient in a divergence form elliptic partial differential equation.
Here there has been substantial work in the numerical analysis community devoted 
to quantifying the error in the solution of the problem in the case where the diffusion 
coefficient is a random field specified through a Karhunen-Lo\`eve or polynomial
chaos expansion which is truncated 
\cite{BNT07,BTZ04,BSw09,CDS10,Cha10,FSwT05,Mat08,MatB99,NTW08a,NTW08b,ST06,ZhL04}.
However in practice the unknown diffusion coefficient is often conditioned 
by observational data, leading to an inverse problem \cite{McT96}.
This gives rise
to a far more complicated measure on the diffusion 
coefficient.
The purpose of this paper is to study this inverse problem and, in particular, 
the effect of approximating the underlying probability measure via
a finite, but large, set of real valued random variables.
Much of the existing numerical analysis concerning 
groundwater flow with random permeability 
requires uniform upper and lower bounds
over the probability space,  and hence
excludes the log normal permeability distributions
widely used in applications. An exception is the 
recent paper \cite{Cha10} in which the log normal
case is studied. 
For the inverse problem we study here we also
use log normal priors which are attractive from an 
inverse modeling perspective
precisley because no prior bounds on the
permeability may be known.
A key tool when working with
log normal distributions, and hence Gaussian measures, 
is the Fernique theorem
which faciltates functional integration of a wide
class of functions, including the exponential of 
quadratics, against Gaussian measures \cite{DaZa92}. 
The paper \cite{Cha10} exemplifies the key role of
the Fernique theorem and it will also be used
extensively in our developments of the inverse problem.

We consider the elliptic equation 
\begin{equation}\label{eq:epde}
\begin{array}{cl}
-\nabla\cdot \lt(\e^{u}\nabla p\rt)=f+\cdiv g,&x\in D,\\
\qquad p=\phi,& x\in\partial D,
\end{array}
\end{equation}
with $D$ an open, bounded and connected subset of  $\R^d$, $d\le 3$, 
$p$, $u$, $f$ and $\phi$ scalar functions and $g$ a vector function on $D$. 
Given any $u\in L^\infty(D)$ we define $\lambda(u)$ and $\Lambda(u)$
by 
$$
\lambda(u)=\mathrm{ess}\inf_{x\in D}\e^{u(x)},\qquad
\Lambda(u)=\mathrm{ess}\sup_{x\in D}\e^{u(x)}.
$$
Where it causes no confusion we will simply write $\lambda$ or $\Lambda$.
Equation (\ref{eq:epde}) arises as a model for flow in a porous medium with $p$ the pressure
(or the head) and $\e^u$ the permeability (or the transmissivity); the velocity $v$
is given by the formula $v\propto -\e^u\nabla p$. 

Consider making noisy observations of a set of
linear functionals $l_j$
of the pressure field $p$, so that $l_j:p\mapsto l_j(p)\in\R$. 
We write the observations as
\begin{equation}
y_j=l_j(p)+\eta_j,\quad j=1,\cdots,K.
\label{eq:obs}
\end{equation}
We assume, for simplicity,
that $\eta=\{\eta_j\}_{j=1}^K$ is a mean zero
Gaussian observational noise with covariance $\Gamma$. 
In this paper we consider $l_j$ to be either:
\begin{itemize}
\item[a)] pointwise evaluation of $p$ at a point $x_j\in D$ 
(assuming enough regularity for $f$, $g$ and $\phi$ so that this makes sense almost everywhere in $D$); or

\item[b)] $l_j:H^1(D)\to\R$, a functional on 
$H^1(D)$ (again assuming enough regularity 
for $f$, $g$ and $\phi$ so that $p\in H^1(D)$).

\end{itemize}

Our objective is to determine $u$ from $y=\{y_j\}_{j=1}^K \in \bbR^K$.
 We adopt a probabilistic approach 
which we now outline.
In the sequel we derive conditions under
which we may view $l_j(p)$ as a function of $u$. 
Then, concatenating the data, we have
$$y={\cal G}(u)+\eta,$$
with
\begin{equation}
\label{eq:bvpa}
\cG(u)=\bigl(l_1(p),\cdots,l_{\kk}(p)\bigr)^T.
\end{equation}
Here the {\it observation operator} $\cG$ maps $X$ into $\bbR^K$
where $X$ is a Banach space which we specify below in
various scenarios and
is determined by the forward model. 
From the properties of $\eta$ we see that the likelihood of the data $y$ given $u$ is
$$\bbP\,(y\mid u) \propto 
\exp \Bigl(-\frac{1}{2}\Big|\Gamma^{-\frac12}\bigl(y-{\cal G}(u)\bigr)\Big|^2\Bigr)$$ %
where $|\cdot|$ is the standard Euclidean norm. 
Let $\bbP(u)$ denote a prior distribution on the function $u$. 
If $u$ were finite dimensional, the posterior distribution, by Bayes' rule, would be given by
$$
\bbP(u|y)\propto \bbP(y|u)\,\bbP(u).
$$
For infinite dimensional spaces however there is no density with respect to the Lebesgue measure. In this context Bayes rule should be interpreted as providing the
Radon-Nikodym derivative between the posterior measure $\mu^y(\ud u)=\bbP(\ud u|y)$ 
and the prior measure $\mu_0(\ud u)=\bbP(\ud u)$:
\begin{equation}\label{eq:radon2}
\frac{d\mu^y}{d\mu_0}(u) \propto 
\exp \Bigl(-\frac{1}{2}\Big|\Gamma^{-\frac12}\bigl(y-{\cal G}(u)\bigr)\Big|^2\Bigr).
\end{equation}
The problem of making sense of Bayes rule for probability measures on function spaces
with a Gaussian prior is addressed in \cite{CDRS09, CDaS10, St10}.
In Section \ref{sec:infinte-Bayes} 
we recall these results, and then, in Subsection \ref{sec:approx}, 
state and prove a new result concerning
stability properties of the posterior measure
with respect to finite dimensional approximation of the prior. 

In Section \ref{sec:estimates} we show that the observation operator $\cG$ 
of the elliptic problem described above satisfies
boundedness and Lipschitz continuity conditions 
for appropriate choices of the Banach space $X$. 
In Section \ref{sec:post} we combine the results of
the preceding two sections to show that formula
(\ref{eq:radon2}) holds for the posterior measure, and
to study its approximation with
respect to finite dimensional specification of
the prior and posterior.
Section \ref{sec:conclu} contains some concluding remarks.

\section{Bayesian approach to inverse problems for functions}
\label{sec:infinte-Bayes}

In this section we recall various theoretical results 
related to the development of Bayesian statistics on 
function space. We also state and prove a new result on the weak
approximation of the posterior using finite
dimensional truncation of the Karhunen-Lo\`eve expansion.
 We assume that we are given two
Banach spaces $X$ and $Y$, a function $\Phi:X \times Y \to \bbR$
and a probibility measure $\mu_0$ supported on $X$.
Consider the putative Radon-Nikodym derivative
\begin{subequations}
\label{eq:radon}
\begin{equation}
\label{eq:radon1}
\frac{d\mu^y}{d\mu_0}(u)= \frac{1}{Z(y)}\exp \bigl(-\Phi(u;y)\bigr),
\end{equation} 
\begin{equation}
Z(y)=\int_X \exp(-\Phi(u;y))\,\mu_0(\ud u).
\label{eq:norma}
\end{equation}
\end{subequations}
Our aim is to find conditions on $\Phi$ and $\mu_0$ under
which $\mu^y$ is a well-defined probability measure on $X$, 
which is continuous in the data $y$, and to describe
an approximation result for $\mu^y$ with respect to
approximation of $\Phi$.
Remarkably these results may all be proved simply by
establishing properties of the operator $\Phi$ and
its approximation on $X$,
and then choosing the prior Gaussian measure so that
$\mu_0(X)=1$.
This clearly separates the analytic and probabilistic aspects 
of the Bayesian formulation of inverse problems for functions.
The results of this section are independent of the specific inverse problem 
described in Section \ref{sec:intro} and have wide applicability.
Note, however, that \eqref{eq:radon2} is a particular case
of the general set-up of \eqref{eq:radon}, with
$Y=\bbR^K$. But the level of generality we adopt allows us to work
with infinite dimensional data (functions) and/or non-Gaussian
observational error $\eta$. 
In particular 
if the data $y=\cG(u)+\eta$
where $\cG:X\to Y$ is the observation operator, $Y$ is
a Hilbert space and $\eta$ 
is a mean zero random field on $Y$ 
with Cameron-Martin space $\big( E,\la\cdot,\cdot\ra_E,\|\cdot\|_{E}\big)$
then we define $\Phi$ as 
\begin{align*}
\Phi(u;y)
&=\frac12\|y-\cG(u)\|_{E}^2-\frac12\|y\|_{E}^2\\
&=\frac{1}{2}\|\cG(u)\|_E^2-\la y,\cG(u)\ra_E.
\end{align*}
On the other hand if $Y=\bbR^K$ and $\eta$ has Lebesgue
density $\rho$, then  we define $\Phi$ by
the identity $\exp(-\Phi(u;y))=\rho(y-\cG(u)).$
Note that these two definitions agree, up to
an additive constant depending only on $y$, when $\eta$ is
Gaussian and $Y$ is finite dimensional; such a constant
simply amounts to adjusting the normalization $Z(y).$
The subtraction of the term $\frac12\|y\|_{E}^2$ in the 
infinite dimensional data setting is required to make
sure that $\Phi(\cdot;y)$ is almost surely finite with respect 
to $\eta$ \cite{St10}.
For simplicity we work in the case where $X$
comprises periodic functions on the $d-$dimensional
torus $\bbT^d$; generalizations are possible.


\subsection{Well-defined and well-posed Bayesian inverse problems}

In \cite{CDRS09, CDaS10, St10}, it is shown that some appropriate properties of the
log likelihood $\Phi$ together with an appropriate choice of 
a Guassian prior measure implies the existence of
a well-posed Bayesian inverse problem. 
Here, we recall these results. To this end
we assume the following conditions on $\Phi$:

\begin{asp} \label{asp1}
Let $X$ and $Y$ be Banach spaces.
The function $\Phi:X\times Y\to\R$ satisfies:
\begin{itemize}
\item[(i)] for every $\epsilon>0$ and $r>0$ there is $M=M(\epsilon,r)\in \R$, such that for all $u\in X$, and for all $y\in Y$ such that $\|y\|_Y<r,$
\begin{equation*}
\Phi(u,y) \geq M -\epsilon\|u\|^2_{X};
\end{equation*}
\item[(ii)] for every $r>0$ there exists $K=K(r)>0$ such that for all $u\in X$, $y\in Y$ with $\max\{\|u\|_{X},\|y\|_Y\}<r$
\begin{equation*}
\Phi(u,y) \leq K;
\end{equation*}
\item[(iii)] for every $r>0$ there exists $L=L(r)>0$ such that for all $u_1,u_2\in X$ and $u\in Y$ with $\max\{\|u_1\|_{X},\|u_2\|_{X},\|y\|_Y\}<r$
\begin{equation*}
|\Phi(u_1,y)-\Phi(u_2,y)| \leq L\|u_1-u_2\|_{X};
\end{equation*}
\item[(iv)] 
for every $\epsilon>0$ and $r>0$, there is $C=C(\epsilon,r)\in\R$
 such that for all 
$y_1,y_2\in Y$ with $\max\{\|y_1\|_Y,\|y_2\|_Y\}<r$ and for every $u\in X$
\begin{equation*}
|\Phi(u,y_1)-\Phi(u,y_2)| \leq \exp(\epsilon\|u\|^2_{X}+C)\|y_1-y_2\|_Y.
\end{equation*}
\end{itemize}
\end{asp}

We now recall two results from \cite{CDRS09} 
concerning well-definedness 
and well-posedness of the posterior measure.

\begin{theorem}\label{t:welldb} \cite{CDRS09}
Let $\Phi$ satisfy Assumptions \ref{asp1}(i)--(iii).
Assume that $\mu_0$ is a Gaussian measure 
with $\mu_0(X=1)$.
Then $\mu^y$ given by (\ref{eq:radon}) is a 
well-defined probability measure.
\end{theorem}

One can also show continuity of the posterior in the 
Hellinger metric $\dhh$ (see \cite{CDaS10} for the 
definition) with respect to the data $y$.
For any two measures $\mu$ and $\mu'$ both absolutely continuous
with respect to the same reference measure, and function
$G: X \to S$ with $S$ a Banach space, 
\begin{align}\label{e:generalE}
\|\bbE^\mu G(u)-\bbE^{\mu'}G(u)\|_S\le C\bigl(\bbE^{\mu}\|G(u)\|_{S}^2+\bbE^{\mu'}\|G(u)\|_{S}^2\bigr)^{\frac12}
\dhh(\mu,\mu').
\end{align}
Theorem \ref{t:wellpb} which follows
is hence quite useful: for example it implies
Lipschitz continuity of the posterior mean with respect to data. 
(See \cite{CDRS09}, Section 2).
\begin{theorem}\label{t:wellpb}
\cite{CDRS09}
Let $\Phi$ satisfy Assumptions \ref{asp1}(i)--(iv).
Assume that $\mu_0$ is a Gaussian measure 
with $\mu_0(X)=1$.  Then
$$
\dhh(\mu^y,\mu^{y'})\le C\,\|y-y'\|_{Y}
$$
where $C=C(r)$ with $\max\{\|y\|_{Y},\|y'\|_{Y}\}\le r$.
\end{theorem}

\subsection{Approximation of the posterior}

In this section we recall a result concerning
approximation of $\muy$
on the Banach space $X$ when the function $\Phi$
is approximated. This will be used in the next
subsection for approximation of $\muy$ on a finite
dimensional space.
Consider $\Phi^N$ to be an approximation of $\Phi$. 
Here we state a result which quantifies the effect of this approximation in the posterior measure in terms of
the aproximation error in $\Phi$. 

Define $\mu^{y,N}$ by
\begin{subequations}
\label{eq:muN}
\begin{equation}
\label{eq:muNa}
\frac{\ud\mu^{y,N}}{\ud\mu_0}(u)
=\frac{1}{Z^N(y)}\exp\bigl(-\Phi^N(u)\bigr),
\end{equation}
\begin{equation}
\label{eq:Anormz2}
Z^N(y)=\int_{X} \exp\bigl(-\Phi^N(u)\bigr) \ud\mu_0(u).
\end{equation}
\end{subequations}
We suppress the dependence of $\Phi$ and $\Phi^N$
on $y$ in this section as it is considered fixed.

%
%
\begin{theorem}
\label{t:wellp2b}
\cite{CDaS10}
Assume that the measures $\mu$ and $\mu^N$ are both absolutely continuous with respect to a Gaussian $\mu_0$ with $\mu_0(X)=1$, and given by (\ref{eq:radon}) 
and (\ref{eq:muN}) respectively. 
Suppose that $\Phi$ and $\Phi^N$ satisfy 
Assumptions \ref{asp1}(i) and (ii), uniformly in $N$,
and that
for any $\epsilon>0$ there exists $C=C(\epsilon)\in\R$ such that
\begin{equation*}
|\Phi(u)-\Phi^N(u)| \leq \exp(\epsilon\|u\|^2_{X}+C)\psi(N)
\end{equation*}
where $\psi(N)\rightarrow 0$ as $N\rightarrow \infty$.
Then there exists a constant independent of $N$ such that
\begin{equation*}
\dhh(\mu, \mu^N)\leq C\psi(N).
\end{equation*}
\end{theorem}

\subsection{Approximating the posterior measure 
in a finite dimensional space}\label{sec:approx}

In this section, we again consider approximation of the
posterior measure for the inverse problem
on $X\subseteq L^2(\TT^d)$. 
But here we additionally assume that the approximation
is made in a finite dimensional subspace
and hence corresponds to something that can
be implemented computationally. 
Our approximation space will be defined by truncating
the Karhunen-Lo\'eve basis $\{\ppsi_l\}_{l=1}^{\infty}$
comprising the eigenfunctions of the covariance
operator of the Gaussian measure $\mu_0.$ 
For simplicity we assume from now on that $\mu_0$ is centred
(zero mean).

Define the subpace 
$W^N$ spanned by the $\{\ppsi_l\}_{l=1}^{N}$ and 
let $W^\bot$ denote the complement of $W^N$
in $L^2(\bbT^d)$. 
For any $u\in X$ let
\begin{align}\label{eq:truncg}
u^N=\sum_{l=1}^N u_l\,\ppsi_l,\quad\mbox{with } u_l=(u,\ppsi_l),
\end{align}
where $(\cdot,\cdot)$ is the $L^2$-inner product.
The approximate posterior measure will induce
a measure on the coefficients $\{u_l\}_{l=1}^N$
appearing in \eqref{eq:truncg}, and hence a measure
$\nu^N$ on $u^N \in W^N.$ Our interest is in
quantifying the error incurred when approximating
expectations under $\muy$ by expectations
under $\nu^N.$

For simplicity we consider the case where the data $y$
is finite dimensional and the posterior measure is
defined via \eqref{eq:radon} with 
\begin{equation}
\label{eq:GF}
\Phi(u):=\frac{1}{2}\Big|\Gamma^{-\frac12}
\bigl(y-{\cal G}(u)\bigr)\Big|^2.
\end{equation}
Here $|\cdot|$ denotes the Euclidean norm and $\Gamma$
is assumed positive and symmetric.
We drop explicit $y$ dependence in $\Phi$ (and approximation
$\Phi^N$) 
throughout this section.

Note that $\mu_0$ factors as the product 
of two independent measures
$\mu_0^N\otimes\mu_0^\bot$ on $W^N\oplus W^\bot$.
Let $P^N$ be the orthogonal projection of $L^2(\TT^d)$ onto $W^N$,
and $P^\bot=I-P^N$.
Any $u\in X$ can be written as $u=u^N+u^\bot$ where
$u^N=P^N u$ is given by \eqref{eq:truncg}
and $u^\bot=P^\bot u$.
We define $\cG^N(\cdot)=\cG(P^N \cdot)$ and
consider the approximate measure \eqref{eq:muN}
with $\Phi^N$ given by
\begin{align}\label{eq:muGN}
\Phi^N(u)=\frac{1}{2}\bigl|\Gamma^{-1/2}(y-\cG^N(u))\bigr|^2.
\end{align}
Because $\cG^N(u)$ depends only on $u^N$,
and because $\mu_0=\mu_0^N\otimes\mu_0^\bot$
the resulting measure $\mu^{y,N}$ on $X$
can be factored as $\mu^{y,N}=\nu^N\otimes\mu^\bot$
where
\begin{align}\label{eq:muGN1}
\frac{\ud\nu^{N}}{\ud\mu_0^N}(u)
\propto \exp\bigl(-\frac{1}{2}|\Gamma^{-1/2}(y-\cG^N(u))|^2\bigr)
\end{align}
and $\mu^\bot=\mu_0^\bot$. The measure $\nu^N$ given by 
(\ref{eq:muGN1}) is finite dimensional and amenable
to statistical sampling techniques, such as MCMC.
For the purposes of this paper we assume 
that expectations with respect to $\nu^N$ on $W^N$ 
can be computed exactly. An overview of techniques
for sampling such measures, 
in a manner robust to increasing
$N$, can be found in \cite{CRSW10}.

We are interested in approximating
expectations under $\mu^y$
of functions $G:X\to S$, $S$ a Banach space.
For example, $G(u)$ may denote the pressure field, or covariance of
the pressure field for the elliptic inverse in Section \ref{sec:intro}.
Abusing notation, we will sometimes
write $G(u)=G(u^N,u^\bot)$.
In practice we are able to compute expectations of $G(u^N,0)$
under $\nu^N$.
Thus we are interested in estimating the weak error
\begin{align}\label{e:werror}
e=\|\bbE^{\mu^y} G(u^N,u^\bot)-\bbE^{\nu^N} G(u^N,0)\|_S.
\end{align}
We now state and prove a theorem concerning this error,
under the following assumptions on $\cG$ and $G$.

\begin{asp} \label{a:approx}
Assume that $X$, $X'$ and $X''$ are Banach spaces, and $X$ is continuously embedded into $X'$ and
$X'$ is continuously embedded into $X''$.
Suppose also that the centred Gaussian probability
measure $\mu_0$ satisfies $\mu_0(X)=1.$
Then, for all $\eps>0$, there is $K=K(\eps) \in (0,
\infty)$, such that, for all $u_1,u_2 \in X$,
\begin{align*}
|\cG(u_1)-\cG(u_2)| & \le K\exp\bigl(\eps\max\{
\|u_1\|_{X}^2,\|u_2\|_{X'}^2\}\bigr)\|u_1-u_2\|_{X''},\\
\|G(u_1)-G(u_2)\|_{S} & \le K\exp\bigl(\eps\max\{
\|u_1\|_{X}^2,\|u_2\|_{X'}^2\}\bigr)\|u_1-u_2\|_{X''}.
\end{align*}
Furthermore $\|P^{\bot}\|_{{\cal L}(X,X'')}=\psi(N)
\to 0$ as $N \to \infty$, and $\|P^{\bot}\|_{{\cal L}(X,X')}$
is bounded independently of $N$.
\end{asp}

\begin{theorem} \label{t:tapp}
Let Assumptions \ref{a:approx} hold and assume that
$\Phi$ and $\Phi^N$ are given by \eqref{eq:GF} and
\eqref{eq:muGN} respectively. Then the probability
measures $\mu^y$ and $\mu^{y,N}$ are
absolutely continuous with respect to $\mu_0$
and given by \eqref{eq:radon} and \eqref{eq:muNa}.
Furthermore the weak error \eqref{e:werror}
satisfies $e \le C\psi(N)$ as $N \to \infty.$
\end{theorem}

\begin{proof}
Since $X\hookrightarrow X'\hookrightarrow X''$ we have,
for some constants $C_1, C_2>0$, $\|\cdot\|_{X''}\le C_1\|\cdot\|_{X'} \le
C_2\|\cdot\|_{X}.$ From the assumptions on $\cG$ and $P^\bot$
we deduce that Assumptions \ref{asp1}(i)--(iii)
hold for $\Phi$ and $\Phi^N$ given by \eqref{eq:GF} 
and \eqref{eq:muGN} respectively. Thus $\mu^y$
and $\mu^{y,N}$ are well-defined probability
measures, both absolutely continuous with respect to
$\mu_0$, and satisfying $\mu^y(X)=\mu^{y,N}(X)=1.$
By the triangle inequality we have $e\le e_1+e_2$ where
\begin{align}
e_1&=\|\bbE^{\mu^y} G(u^N,u^\bot)-\bbE^{\mu^{y,N}} G(u^N,u^\bot)\|_S,\label{e:e1}\\
e_2&=\|\bbE^{\mu^{y,N}} G(u^N,u^\bot)-\bbE^{\nu^{N}} G(u^N,0)\|_S.\nonumber
\end{align}

We first estimate $e_1$. Note that, by use
of the assumptions on $\cG$ and $\cG^N$, we have
that for any $\eps>0$ there is $K=K(\eps) \in (0,
\infty)$ such that
\begin{align*}
|\Phi(u)-\Phi^N(u)| &\le \frac12
\bigl|\Gamma^{-\frac12}\bigl(2y-\cG(u)-\cG^N(u)\bigr)\bigr|
\bigl|\Gamma^{-\frac12}\bigl(\cG(u)-\cG^N(u)\bigr)\bigr|\\
&\le K\exp\bigl(\eps\|u\|_{X}^2\bigr)\psi(N).
\end{align*}
By Theorem \ref{t:wellp2b} we deduce that the Hellinger
distance between $\mu^y$ and $\mu^{y,N}$ tends
to zero like $\psi(N)$ and hence, by \eqref{e:generalE},
that $e_1 \le C\psi(N)$. This last bound
follows after noting that the required integrability
of $\|G(u)\|_{S}^2$ and $\|G(u^N)\|_{S}^2$
follows from the Lipschitz
bound on $G$, the operator norm bound on $P^{\bot}$
and the Fernique theorem \cite{DaZa92}. 

We now estimate $e_2$.
Because $\bbE^{\mu^{y,N}} G(u^N,0)=\bbE^{\nu^N}G(u^N,0)$
we obtain 
\begin{align}\label{e:e2}
e_2\le \bbE^{\mu^{y,N}}\|G(u)-G(u^N)\|_S
\end{align}
From the Lipscthiz properties of $G$ we deduce that
\begin{align*}
e_2 & \le K_1(\eps) \bbE^{\mu^{y,N}} \Bigl(\exp\bigl(\eps\max\{
\|u\|_{X}^2,\|P^N u\|_{X'}^2\}\bigr)\|P^{\bot}u\|_{X''}\Bigr)\\
& \le K_2(\eps) \bbE^{\mu^{y,N}} \Bigl(\exp\bigl(C\eps
\|u\|_{X}^2\bigr)\|u\|_{X}\Bigr)\psi(N)\\
& \le K_3(\eps) \bbE^{\mu^{y,N}} \Bigl(\exp\bigl(2C\eps
\|u\|_{X}^2\bigr)\Bigr)\psi(N)
\end{align*}
where $C$ is independent of $\epsilon$.
The result follows by the Fernique theorem.
\end{proof}

The results of this section are quite general,
concerning a wide class of Bayesian inverse
problems for functions, using Gaussian priors.
These results can also be generalized to the case of 
Besov priors introduced in \cite{Las09} and will be
presented elsewhere.

In the next section we establish conditions which enable
application of Theorems \ref{t:welldb}, \ref{t:wellpb},
\ref{t:wellp2b} and \ref{t:tapp} to the 
specific elliptic inverse problem 
described in Section \ref{sec:intro}.
We then apply these results in section \ref{sec:post}. 
%


\section{Estimates on the observation operator}\label{sec:estimates}

In order to apply Theorems 
\ref{t:welldb}, \ref{t:wellpb}, \ref{t:wellp2b} and \ref{t:tapp}
to the elliptic inverse problem described
in section \ref{sec:intro}
we need to prove certain properties of 
the observation operator $\cG$ given by (\ref{eq:bvpa}),
viewed as a mapping from a Banach space $X$ into $\R^m$,
and the function $G:X\to S$.
Then the prior measure $\mu_0$ must be chosen so that $\mu_0(X)=1$.
It is hence desirable to find spaces $X$ which are as large as possible,
so as not to unduly restrict the prior, but for which the desired properties of the observation operator hold. 
As discussed in section \ref{sec:intro} we consider the observation operator 
obtained from either the pointwise measurements of $p$ or bounded linear functionals
of $p\in H^1(D)$. 
We obtain the bounds on the observation operator for each of these cases
in the following.

\subsection{Measurements from bounded linear functionals of $p\in H^1(D)$ }

In this case, using standard energy estimates, 
we have the following result: 

\begin{proposition}\label{pr:Gbnd0}
Consider equation (\ref{eq:epde}) with $D\subset \R^d$, $d=2,3$ a bounded domain, the boundary of $D$, $\partial D$, $C^1$-regular,
$f\in H^{-1}(D)$, $g\in L^{2}(D)$ 
and assume that $\phi$ may be extended to a function 
$\phi\in H^1(D)$. Then there exists a constant 
$C=C(\|f\|_{H^{-1}},\|g\|_{L^{2}},\|\phi\|_{H^1(D)})$ 
such that:
\begin{itemize}
\item[i)] if $u \in L^{\infty}(D)$ then
$$
\|p\|_{H^1(D)}\,\le\, C\, \exp(2\|u\|_{L^\infty(D)});
$$
\item[ii)] 
 if $u_1,u_2\in L^\infty(D)$, and $p_1,p_2$ are the
corresponding solutions of (\ref{eq:epde}), then
$$
\|p_1-p_2\|_{H^1(D)}\le C\,\exp(4\max\{\|u_1\|_{L^\infty},\|u_2\|_{L^\infty}\})\|u_1-u_2\|_{L^\infty}.
$$
\end{itemize}
\end{proposition}

\begin{corollary}\label{c:G1}
Let the assumptions of Proposition \ref{pr:Gbnd0} hold.
Consider $\cG(u)=(l_1(p),\dots,l_{\kk}(p))$ with $l_j$, $j=1,\dots,K$ bounded linear functionals on $H^1(D)$. 
Then for any $u,u_1,u_2\in L^\infty(D)$ we have
$$
|\cG(u)|\le C\,\exp(2\,\|u\|_{L^\infty(D)}),
$$
and
$$
|\cG(u_1)-\cG(u_2)|\,
\le\, C\,\exp(4\max\{\|u_1\|_{L^\infty},\|u_2\|_{L^\infty}\})\,\|u_1-u_2\|_{L^\infty(D)},
$$
with $C=C(K,\|f\|_{H^{-1}},\|g\|_{L^{2}},\|\phi\|_{H^1(D)},
{\rm{max}}_j\|l_j\|_{H^{-1}})$.
\end{corollary}

\begin{proof}[of Proposition \ref{pr:Gbnd0}]
{\it i}) Substituting $q=p-\phi$ in (\ref{eq:epde}) and taking the inner product with $q$, 
we obtain
\begin{align*}
\e^{-\|u\|_{L^\infty}} \|\nabla q\|_{L^2}\,
\le\, \e^{\|u\|_{L^\infty}}\,\|\nabla\phi\|_{L^2}+\|f\|_{H^{-1}}+\|g\|_{L^2}
\end{align*}
and therefore
\begin{align}\label{e:pu}
\|\nabla p\|_{L^2}
\le (1+\e^{2\|u\|_{L^\infty}})\,\|\nabla\phi\|_{L^2}
+\e^{\|u\|_{L^\infty}}(\|f\|_{H^{-1}}+\|g\|_{L^2})
\end{align}
which implies the result of part ({\it i}).

{\it ii}) 
The difference $p_1-p_2$ satisfies
$$
\nabla\cdot({\e}^{u_1}\nabla(p_1-p_2))=\nabla\cdot((\e^{u_2}-\e^{u_1})\nabla p_2).
$$
Taking the inner product of the equation with $p_1-p_2$ gives
\begin{align*}
\e^{-\|u_1\|_{L^\infty}}\|\nabla(p_1-p_2)\|_{L^2}^2
&\le \|e^{u_2}-e^{u_1}\|_{L^\infty}\|\nabla p_2\|_{L^2}\,\|\nabla(p_1-p_2)\|_{L^2}.
\end{align*}
For any $x\in D$ we can write
\begin{align*}
\bigl|\e^{u_2(x)}-\e^{u_1(x)}\bigr|
&=\Bigl|\big(u_2(x)-u_1(x)\big)\int_0^1 \e^{su_2+(1-s)u_1}\,\ud s\Bigr|\\
&\le \|u_2-u_1\|_{L^\infty}\,\e^{\max\{\|u_1\|_{L^\infty},\|u_2\|_{L^\infty}\}}.
\end{align*}
Hence using the estimate for $\|\nabla p_2\|_{L^2}$ from part (i), 
the result follows.
\end{proof}

\subsection{Pointwise measurements of $p$}
Here we obtain bounds on the
$L^\infty$-norm of the pressure $p$.

\begin{proposition}\label{pr:Gbnd}
Consider equation (\ref{eq:epde}) with $D\subset \R^d$, $d=2,3$ a bounded domain, the boundary of $D$, $\partial D$, $C^1$-regular, 
$f\in L^r(D)$, $g\in L^{2r}(D)$ with $r>d/2$ 
and $\phi\in L^\infty(\partial D)$. 
There exists $C=C(K,d,r,D,\sup_{\paro}|\phi|,\|f\|_{L^r},\|g\|_{L^{2r}})$ such that
$$
\|p\|_{L^\infty(D)}\,\le\, C\, \exp(\|u\|_{L^\infty(D)}).
$$
\end{proposition}

%
%

\begin{proposition}\label{pr:GLip}
Let the assumptions of Proposition \ref{pr:Gbnd} hold. 
Suppose also that
$\phi$ may be extended to a function $\phi\in W^{1,2r}(D)$. Then: 
\begin{itemize}
\item[i)] the mapping $u\mapsto p$ from $C(\bar D)$ into $L^\infty(D)$ is locally Lipschitz continuous;

\item[ii)] assume that $u_1\in L^\infty(D)$, $u_2\in C^t(D)$ for some $t>0$, and $p_1,p_2$ are the
corresponding solutions of (\ref{eq:epde}).
Then for any $\epsilon>0$ there exists 
$C$ depending on $K,d,t,\epsilon,\|f\|_{L^r},\|g\|_{L^{2r}}$ and $\|\phi\|_{W^{1,2r}}$ such that
\begin{align*}
\|p_1-p_2\|_{L^\infty(D)}
\le  C\,
\exp\left(c_0
      \max\{\|u_1\|_{L^\infty(D)},\|u_2\|_{C^t(D)}\}
              \right)\|u_1-u_2\|_{L^\infty(D)}
\end{align*}
where $c_0=4+(4+2d)/t+\epsilon$.
\end{itemize}
\end{proposition}

\begin{corollary}\label{c:G2}
Let the assumptions of Proposition \ref{pr:Gbnd} hold.
Consider $\cG(u)=(l_1(p),\dots,l_{\kk}(p))$ with $l_j$, $j=1,\dots,K$ pointwise evaluations of $p$ at $x_j\in D$. Then for any $u\in L^\infty(D)$ there exists $C=C(K,d,r,D,\sup_{\paro}|\phi|,\|f\|_{L^r},\|g\|_{L^{2r}})$ such that
$$
|\cG(u)|\le C\,\exp(\|u\|_{L^\infty(D)}).
$$
If $u_1\in L^\infty(D)$, $u_2\in C^t(D)$ for some $t>0$ and $\phi$ can be extended to a function $\phi\in W^{1,2r}(D)$ then,
for any $\epsilon>0$,
\begin{align*}
|\cG(u_1)-\cG(u_2)|
\le  C\,\exp\left(c_0
      \max\{\|u_1\|_{L^\infty(D)},\|u_2\|_{C^t(D)}\}
              \right)\|u_1-u_2\|_{L^\infty(D)}
\end{align*}
with $C=C(K,d,t,\epsilon,\|f\|_{L^r},\|g\|_{L^{2r}},\|\phi\|_{W^{1,2r}})$
and  $c_0=4+(4+2d)/t+\epsilon$ for any $\epsilon>0$. 
\end{corollary}

%
%
\begin{proof}[of Proposition \ref{pr:Gbnd}]
By Theorem 8.15 and 8.17 of \cite{GT83} we have,
recalling the definition of $\lambda$,
\begin{align*}
\sup_{D}|p|
&\le \sup_{\paro}|\phi|+ C(d,r,|D|)
     \,\lt(\|p_0\|_{L^2}+\frac{1}{\lambda}\|f\|_{L^r}+\frac{1}{\lambda}\|g\|_{L^{2r}}\rt)
\end{align*}
where $p_0$ is the solution of (\ref{eq:epde}) with $\phi=0$.
Taking the inner product of (\ref{eq:epde}) with $p_0$  
we obtain, for $c_i=c_i(d,r,|D|)$,
\begin{align*}
\|\nabla p_0\|_{L^2}&\le \frac{c_1}{\lambda}\,(\|f\|_{H^{-1}}+\|g\|_{L^2})\\
&\le \frac{c_2}{\lambda}\,(\|f\|_{L^r}+\|g\|_{L^2}).
\end{align*}
Hence
\begin{align}\label{eq:psup}
\sup_D|p|\le \frac{1}{\lambda}C(d,r,D)\,(\sup_{\partial D}|\phi|+\|f\|_{L^r}+\|g\|_{L^{2r}})
\end{align}
and the result follows since $\lambda\ge \e^{-\|u\|_{L^\infty(D)}}$.
\end{proof}

%
%
\begin{proof}[of Proposition \ref{pr:GLip}]
 The difference $p_1-p_2$ satisfies
\begin{align}\label{eq:q}
\begin{array}{rl}
\nabla\cdot \lt(\e^{u_1}\nabla (p_1-p_2)\rt)=\cdiv F,&x\in D,\\
p_1(x)=p_2(x),& x\in\partial D,
\end{array}
\end{align}
where
$$
F=- (e^{u_1(x)}-e^{u_2(x)})\nabla p_2 .
$$
Let 
$$
\lambda_m=\min\{\lambda(u_1),\lambda(u_2)\}\quad\mbox{ and } 
\Lambda_m=\max\{\Lambda(u_1),\Lambda(u_2)\}.
$$
By (\ref{eq:psup}) we have
\begin{align}\label{eq:supq}
\sup_{D} |p_1-p_2|
\le \frac{C(d,r,D)}{\lambda_m} \|F\|_{L^{2r}},
\end{align}
for $r>d/2$. Now, $\|F\|_{L^{2r}}$ may be estimated as follows, 
using Theorem \ref{t:est},
\begin{align*}
\|F\|_{L^{2r}}
&\le \|(\e^{u_1}-\e^{u_2})\,\nabla p_2\|_{L^{2r}}\\
&\le \Lambda_m\,\|u_1-u_2\|_{L^\infty}\,\|\nabla p_2\|_{L^{2r}}\\
&\le \Lambda_m\,C(d,r,D,u_2)
(\|f\|_{L^{r}}+\|g\|_{L^{2r}}+\|\phi\|_{W^{1,2r}})\,\|u_1-u_2\|_{L^\infty(D)}
\end{align*}
where in the last line we have used 
the fact that $2rd/(2r+d)<r$. This gives the Lipschitz continuity. 

To show the second part, we use (\ref{eq:supq}) and 
Corollary \ref{c:est} to write
\begin{align*}
\|p_1-p_2\|_{L^\infty(D)}
&\le \frac{1}{\lambda_m}C(d,r,D)\|F\|_{L^{2r}}\\
&\le \frac{\Lambda_m}{\lambda_m}C(d,r,D)\|\nabla p_2\|_{L^{2r}}\,\|u_1-u_2\|_{L^\infty(D)}\\
&\le  C(d,D,r,t)\frac{\Lambda_m(1+\Lambda_m)}{\lambda_m^2}
        \Big(1+(\Lambda_m/\lambda_m)^\frac{1+d}{t}\|u_2\|_{C^t(D)}^\frac{1+d}{t}\Big)\\
&\qquad\quad\times\Big(1+(\Lambda_m/\lambda_m)^{\frac{1}{t}}\|u_2\|_{C^t(D)}^\frac{1}{t}\Big)\\
&\qquad\quad\times(\|f\|_{L^{r}}+\|g\|_{L^{2r}}+\|\phi\|_{W^{1,2r}})\,\|u_1-u_2\|_{L^\infty(D)}.
\end{align*}
Since $1/\lambda_m,\Lambda_m\le \exp\{ \|u_1\|_{C^t},\|u_2\|_{C^t} \}$ and 
for any $\epsilon>0$ there exists $c=c(\epsilon,t)$ such that
\begin{align*}
&\frac{\Lambda_m(1+\Lambda_m)}{\lambda_m^2}
        \Big(1+(\Lambda_m/\lambda_m)^\frac{1+d}{t}\|u_2\|_{C^t(D)}^\frac{1+d}{t}\Big)
        \Big(1+(\Lambda_m/\lambda_m)^{\frac{1}{t}}\|u_2\|_{C^t(D)}^\frac{1}{t}\Big)\\
&\quad\le\, C(\epsilon)\exp\left(c_0(t)\max\{ \|u_1\|_{L^\infty},\|u_2\|_{C^t} \} \right)
\end{align*}
with  $c_0=4+(4+2d)/t+\epsilon$. The result follows.
\end{proof}

We can now summarize our assumptions on the forcing function 
and boundary conditions of (\ref{eq:epde}) for our two choices of the observation operator $\cG$ as follows. We will use these 
assumptions in subsequent sections. In both cases covered
by these assumptions it is a consequence that $\Phi$
and $\cG$ satisfy Assumptions \ref{asp1} and \ref{a:approx}. 

\begin{asp}\label{asp:Gform}
We consider the observation operator $\cG(u)=(l_1(p),\dots,l_{\kk}(p))$ 
defined as in (\ref{eq:bvpa}) with $D\subset \bbR^d$ a
bounded open set with $C^1$ regular boundary. Then
either: 

\begin{itemize}

\item[i)] the mapping $l_j$ is a bounded linear functional on $H^1(D)$ for $j=1,\dots,K$ and
we assume that $f\in H^{-1}(D)$, $g\in L^2(D)$ and 
$\phi$ may be extended to $\phi\in H^1(D)$; or

\item[ii)] the mapping $l_j$ is the pointwise 
evaluation of $p$ at a point $x_j\in D$ 
for $j=1,\dots,K$ and
we assume that $f\in L^r(D)$, $g\in L^{2r}(D)$ and 
$\phi$ may be extended to  $\phi\in W^{1,2r}(D)$, with $r>d/2$.

\end{itemize}
\end{asp}


\section{Properties of the posterior measure for the
elliptic inverse problem}\label{sec:post}

We use the properties of the observation operator
$\cG$ for the elliptic problem to establish
well-posedness of the inverse problem, and to
study approximation of the posterior measure
in finite dimensional spaces defined via
Fourier truncation.

\subsection{Well-posedness of the posterior measure}
\label{ssec:wd}

Here we show the well-definedness of the posterior measure and its continuity with respect to the data for the elliptic problem. 
We have the following theorem:

%
%

\begin{theorem}\label{thm:dwellposterior}
Consider the inverse problem for finding $u$ 
from noisy observations of $p$ in the form of (\ref{eq:obs}) 
and with $p$ solving (\ref{eq:epde}) in $D=\TT^d$. Let Assumption \ref{asp:Gform} hold and
consider $\mu_0$ to be distributed as $\cN(0,(-\Delta)^{-s})$ with $\Delta$ the Laplacian operator acting on $H^2(\TT^d)$ Sobolev functions with zero average on $\TT^d$,
and $s>d/2$.
Then the measure $\mu(\ud u|y)=\bbP(\ud u|y)$ is absolutely continuous
with respect to $\mu_0$ with Radon-Nikodym derivative 
given by \eqref{eq:radon}, \eqref{eq:GF}. 
Furthermore, the posterior measure is continuous in the Hellinger metric with respect to the data:
$$
\dhh(\mu^{y},\mu^{y'})\le C|y-y'|.
$$
\end{theorem}

\begin{proof}
Let $X=C^{t}(D)$. 
We first define measures $\pi_0$ and $\pi$
on $X\times \R^{K}$ as follows.
Define $\pi_0(\ud u,\ud y)=\mu_0(\ud u)\otimes \rho(\ud y)$
where $\rho$ is a centred Gaussian
with covariance matrix $\Gamma$, assumed positive.
Since $\cG:X\to\R^K$ is continuous and since 
Lemma 6.25 of \cite{St10} shows that
$\mu_0(X)=1$ for any $t \in (0,s-d/2)$,
we deduce that $\cG$ is $\mu_0$ measurable
for $t \in (0,s-d/2)$. 
We define $\rho(\ud y|u)=\cN(\cG(u),\Gamma)$
and then by $\mu_0$-measurability of $\cG$,
$\pi(\ud u,\ud y)=\rho(\ud y|u)\mu_0(\ud u)$.
From the properties of Gaussian on $\R^K$
we have
\begin{align*}
\frac{\ud \pi}{\ud\pi_0}(u;y)\propto \exp(-\Phi(u;y))
\end{align*}
with $\Phi(u;y)=\frac{1}{2}|\Gamma^{-1/2}(y-\cG(u))|^2$.
We now show that Assumptions \ref{asp1} (i)--(iv) hold for 
this $\Phi$.
Assumption \ref{asp1}(i) is automatic becuase 
$$\Phi(u;y):=\frac{1}{2}
\Big|\Gamma^{-\frac12}\bigl(y-{\cal G}(u)\bigr)\Big|^2 \ge 0.$$
By Corollaries \ref{c:G1} and \ref{c:G2},
$\cG$ is bounded on bounded sets in $X$,
and Lipschitz on $X$, for any $t>0$, proving
Assumptions \ref{asp1} (ii), (iii). 
To prove Assumption \ref{asp1} (iv) note that
\begin{align*}
|\Phi(u;y_1)-\Phi(u;y_2)| &\le \frac12
\bigl|\Gamma^{-\frac12}\bigl(y_1+y_2-\cG(u)\bigr)\bigr|
\bigl|\Gamma^{-\frac12}\bigl(y_1-y_2\bigr)\bigr|\\
&\le c_1\exp\bigl(c_2\|u\|_{X}\bigr)|y_1-y_2|
\end{align*}
and
$$\exp(c\|u\|_{X}) \le c_1(\eps) \exp(\eps\|u\|_{X}^2).$$
By Assumption \ref{asp1} (ii) we deduce that
$$
\int_X \exp(-\Phi(u;y))\,\mu_0(\ud y)
\ge \exp(-K(r))\,\mu_0(\|u\|_X<r)>0.
$$
By conditioning (see Lemma 5.3 in \cite{HaiS07} or Section 10.2 in \cite{Dud02})
we deduce that the regular conditional probability 
$\mu^y(\ud u)=\mathbb{P}^\pi(\ud u|y)$ is
absolutely continuous with respect to $\mu_0$
and is given by (\ref{eq:radon2}).
Continuity in the Hellinger metric follows from Theorem \ref{t:wellpb}.
\end{proof}

In the case of a general bounded domain $D$  the prior measure 
$\mu_0\sim\cN(0,\cA^{-s})$ where $s>d-1/2$ and $\cA$ is the Laplacian operator  
acting on $D(\cA)=\{u\in H^2(D):\nabla u\cdot {\bf n}|_{\partial D}=0 \mbox{ and } \int_D u=0\}$,
satisfies $\mu_0(C^t(D))=1$. Hence the techniques of the
preceding theorem can be used to show 
wellposedness of the posterior in this case provided
$s>d-\frac12$.

\subsection{Weak approximation in a Fourier basis}
\label{ssec:f}

Let $x=(x_1,\dots,x_d)\in\mathbb{T}^d$. 
With $\ppsi_{k_1,\dots,k_d}=\e^{i(k_1x_1+\dots+k_dx_d)}$, 
the set $\{\ppsi_{k_1,\dots,k_d}\}_{k_1,\dots,k_d\in\Z}$ forms a basis for $L^2(\TT^d)$.
Define 
$$W^N=\mathrm{span}\{\ppsi_{k_1,\dots,k_d},|k_1|\le N,\dots,|k_d|\le N\}$$ 
and recall the notation established in subsection
\ref{sec:approx}.\footnote{The slightly different
interpretation of $N$ should not cause any 
confusion in what follows.} 
We let $p$ denote the $H^1-$valued random
variable found from solution of the
elliptic problem \eqref{eq:epde}
with $u$ distributed according
to the posterior $\mu^y$ and $p^N$ the
analogous random variable with $u=u^N_{F}$ distributed
according to $\nu^N$ with the Fourier truncation
described above.
We have the following approximation theorem:

\begin{theorem}\label{t:appF}
Consider the inverse problem of finding $u$ 
from noisy observations of $p$ in the form of (\ref{eq:obs}) 
and with $p$ solving (\ref{eq:epde}). 
Let Assumption \ref{asp:Gform} hold and
consider $\mu_0$ to be distributed as $\cN(0,(-\Delta)^{-s})$ with $\Delta$ the Laplacian operator acting on $H^2(\TT^d)$ Sobolev functions with zero average on $\TT^d$,
and $s>d/2$.
Then, for any $t<s-\frac{d}{2}$,
$$
\|\bbE^{\mu^y}p-\bbE^{\nu^{N}}p^N\|_{H^1}\le C\,N^{-t},
$$
and, with $\bar{p}=\bbE^{\mu^y} p$, $\bar{p}^N=\bbE^{\nu^{N}} p^N$
and $S=\mathcal{L}(H^1(\TT^d),H^1(\TT^d))$ we have
$$
\|\bbE^{\mu^y}(p-\bar{p})\otimes (p-\bar{p})
-\bbE^{\nu^{N}}(p^N-\bar{p}^N)\otimes (p^N-\bar{p}^N)\|_S
\le C\,N^{-t}.
$$
\end{theorem}

Note that the mean and covariance of the pressure field,
when conditioned by the data, may be viewed as giving
a quantification of uncertainty in the pressure 
when conditioned on data. 
The previous theorem thus estimates errors arising
in this quantification of uncertainty 
when Karhunen-Lo\`eve truncation is used to represent the
posterior measure.

\begin{proof}[of Theorem \ref{t:appF}]
We apply Theorem \ref{t:tapp} with $G(u)=p$
and $S=H^1(D)$ 
or $G(u)=p\otimes p$ and $S={\cal L}(H^1(\TT^d),H^1(\TT^d)).$
We choose $X'=X''=L^{\infty}$ and $X=C^t$ for any
$t<s-\frac{d}{2}.$ Then $X$ is continuously
embedded into $X'$ and $X'$ into $X''$ and, under the assumptions 
of the theorem, $\mu_0(X'')=\mu_0(X')=\mu_0(X)=1.$
From Proposition \ref{pr:Gbnd0}
it is strightforward to show
the required Lipschitz condition on $G$
(in either the pressure or pressure covariance
cases)
whilst the required Lipschitz condition on $\cG$
follows from Assumptions \ref{asp:Gform}. 
It remains to prove the operator norm bounds on
$P^{\bot}.$ We consider dimension $d=2$ first.
We write
\begin{align*}
u^{N}_F(x,y)
&=\sum_{-N\le k\le N}\sum_{-N\le j\le N} \hat u(j,k)\e^{i(xj+ky)}\\
&=\sum_{-N\le k\le N}\sum_{-N\le j\le N} 
    \frac{1}{4\pi^2}\int_{-\pi}^\pi\int_{-\pi}^\pi u(\zeta,\xi) 
         \e^{-i(\zeta j+\xi k)}\,\ud \zeta\,\ud \xi\, \e^{i(xj+ky)}\\
&=\frac{1}{\pi^2}\int_{-\pi}^\pi\int_{-\pi}^\pi u(\zeta,\xi)\, 
     D_N(x-\zeta)\, D_N(y-\xi)\,\ud \zeta\,\ud \xi\\
&=\frac{1}{\pi^2}\int_{-\pi}^\pi\int_{-\pi}^\pi u(x-\zeta,y-\xi)\, 
     D_N(\zeta)\, D_N(\xi)\,\ud \zeta\,\ud \xi
\end{align*}
where \cite{Zyg88}
\begin{align}\label{eq:DN}
D_N(x)=\frac{1}{2}\sum_{-N\le n\le N}\e^{ixn}=\frac{1}{2}\,\frac{\sin(N+1/2)x}{\sin(x/2)}.
\end{align}
Noting that 
$\int_{-\pi}^\pi D_N(x)\,\ud x=\pi$,
we have
\begin{align*}
u^{N}_F(x,y)&-u(x,y)
=\frac{1}{\pi^2}\int_{-\pi}^\pi\int_{-\pi}^\pi \Big( u(x-\zeta,y-\xi)-u(x,y) \Big)\, 
     D_N(\zeta)\, D_N(\xi)\,\ud \zeta\,\ud \xi\\
&=\frac{1}{\pi^2}\int_{-\pi}^\pi\int_{-\pi}^\pi \Big( u(x-\zeta,y-\xi)-u(x-\zeta,y)\\
&\qquad\qquad\qquad
+u(x-\zeta,y)-u(x,y) \Big)\, 
     D_N(\zeta)\, D_N(\xi)\,\ud \zeta\,\ud \xi\\
 &=\frac{1}{\pi^2}\int_{-\pi}^\pi\left(\int_{-\pi}^\pi \Big( u(x-\zeta,y-\xi)-u(x-\zeta,y) \Big)\, 
     D_N(\xi)\,\ud\xi\right) D_N(\zeta)\,\ud \zeta\\
&\qquad+\frac{1}{\pi^2}\int_{-\pi}^\pi \Big( u(x-\zeta,y)-u(x,y) \Big)\, 
     D_N(\zeta)\,\ud \zeta\,\int_{-\pi}^\pi D_N(\xi)\,\ud \xi.
\end{align*}
To estimate the right-hand side we set 
$f(\xi)=u(x-\zeta,y-\xi)-u(x-\zeta,y)$ for fixed $x,\zeta,y\in\TT$, and noting that $f(\xi)$ is periodic we write 
\begin{align*}
2\int_{-\pi}^\pi f(\xi)\,D_N(\xi)\,\ud\xi
&= \int_{-\pi}^\pi \frac{f(\xi)}{\sin(\xi/2)}\, \sin\left((N+1/2)\xi\right)\,\ud\xi\\
&= -\int_{-\pi-\frac{\pi}{N+1/2}}^{\pi-\frac{\pi}{N+1/2}}
            \frac{f(\xi+\frac{\pi}{N+1/2})}{\sin(\frac{\xi}{2}+\frac{\pi}{2N+1})}
                \sin\left((N+1/2)\xi\right)\,\ud\xi\\
&= -\int_{-\pi}^{\pi}\frac{f(\xi+\frac{\pi}{N+1/2})}{\sin(\frac{\xi}{2}+\frac{\pi}{2N+1})}
                \sin\left((N+1/2)\xi\right)\,\ud\xi\\
&=\frac{1}{2}\int_{-\pi}^{\pi}\left(\frac{f(\xi)}{\sin(\xi/2)}
                -\frac{f(\xi+\frac{\pi}{N+1/2})}{\sin(\frac{\xi}{2}+\frac{\pi}{2N+1})}\right)
                \sin\left((N+1/2)\xi\right)\,\ud\xi
\end{align*}
In the third line we have used 
$\int_{-\pi-\alpha}^{-\pi}g(\xi)\,\ud\xi=\int_{\pi-\alpha}^{\pi}g(\xi)\,\ud\xi$ 
for any $2\pi$ periodic function $g$ and any $\alpha\in\R$.
Let $h=\frac{2\pi}{2N+1}$. We can write
\begin{align*}
4\int_{-\pi}^\pi f(\xi)\,D_N(\xi)\,\ud\xi
&\le \int_{-\pi}^\pi \left|\frac{f(\xi)}{\sin(\xi/2)}
                -\frac{f(\xi+h)}{\sin(\frac{\xi+h}{2})}\right|\,\ud\xi\\ 
&=I_1+I_2+I_3
\end{align*}
with $I_1=\int_{0}^\pi$, $I_2=\int_{-\pi}^{-h}$ and $I_3=\int_{-h}^0$ of the integrand in the right-hand side of the above inequality.
Noting that $f(0)=0$, we have
\begin{align*}
I_1
&= \int_0^\pi \left| \frac{f(\xi)-f(\xi+h)}{\sin(\frac{\xi+h}{2})}
             +f(\xi)\left(\frac{1}{\sin(\xi/2)}-\frac{1}{\sin(\frac{\xi+h}{2})} \right)\right|\,\ud\xi\\
&\le c\,h^t\|f\|_{C^t}\int_0^\pi  \frac{1}{\xi+h}\,\ud\xi
         +\int_0^\pi\frac{\|f\|_{C^t}\,\xi^t(\sin(\xi/2+h/2)-\sin(\xi/2))}
                   {\sin(\xi/2)\,\sin(\xi/2+h/2)}\,\ud\xi\\
&\le c\,h^t\|f\|_{C^t}\log\frac{1}{h}                   
            +c\,h\int_0^\pi\frac{\|f\|_{C^t}\,\xi^t}{\xi\,(\xi+h)}\,\ud\xi\\
&\le c\,\,h^t\|f\|_{C^t}\log\frac{1}{h}
            +c\,h\|f\|_{C^t}\int_0^h\frac{1}{\xi^{1-t}\,h}
            +c\,h\|f\|_{C^t}\int_h^\pi\frac{1}{h^{1-t}\,(\xi+h)}\\
 &\le c\,h^t\|f\|_{C^t}\,\log\frac{1}{h}
         + c\,\,h^t\,\|f\|_{C^t}.  
\end{align*}
Similarly
\begin{align*}
I_2
&=\int_{-\pi}^{-h} \left| \frac{f(\xi)-f(\xi+h)}{\sin(\frac{\xi}{2})}
             -f(\xi+h)\left(\frac{1}{\sin(\frac{\xi+h}{2})}-\frac{1}{\sin(\xi/2)} \right)\right|\,\ud\xi\\
&\le c\,h^t\|f\|_{C^t}\int_{-\pi}^{-h}\frac{1}{-\xi}\,\ud\xi
        +\int_{-\pi}^{-h}\frac{\|f\|_{C^t}\,|\xi+h|^t|\sin(\xi/2)-\sin((\xi+h)/{2})|}
                   {|\sin(\xi/2)|\,|\sin((\xi+h)/2)|}\,\ud\xi\\
&\le c\,h^t\|f\|_{C^t}\log\frac{1}{h}                   
            +c\,h\int_{-\pi}^{-h}\frac{\|f\|_{C^t}\,|\xi+h|^t}{|\xi|\,|\xi+h|}\,\ud\xi\\
&\le c\,\,h^t\|f\|_{C^t}\log\frac{1}{h}
         + c\, h\,\|f\|_{C^t}\int_{-\pi}^{-2h}\frac{1}{|\xi|\,h^{1-t}}\,\ud\xi
         + c\, h\,\|f\|_{C^t}\int_{-2h}^{-h}\frac{1}{h\,|\xi+h|^{1-t}}\,\ud\xi\\
 &\le c\,h^t\|f\|_{C^t}\,\log\frac{1}{h}
         + c\,\,h^t\,\|f\|_{C^t}.              
\end{align*}
Finally
\begin{align*}
I_3
&\le c\,\int_{-h}^0 \left|\frac{f(\xi)}{\xi}\right|+\left|\frac{f(\xi+h)}{\xi+h}\right|\,\ud\xi\\
&\le c\,\|f\|_{C^t}\int_{-h}^0\frac{1}{|\xi|^{1-t}}\,\ud\xi
        +c\,\|f\|_{C^t}\int_{-h}^0\frac{1}{|\xi+h|^{1-t}}\,\ud\xi\\\
&\le c\,h^t\,\|f\|_{C^t}.
\end{align*}
Hence we have
$$
\int_{-\pi}^\pi \Big( u(x-\zeta,y-\xi)-u(x-\zeta,y) \Big)\, 
     D_N(\xi)\,\ud\xi \le C\,N^{-t}\|u\|_{C^t}\,\log N,
$$
and
$$
\int_{-\pi}^\pi \Big( u(x-\zeta,y)-u(x,y) \Big)\, 
     D_N(\zeta)\,\ud \zeta \le C\,N^{-t}\|u\|_{C^t}\,\log N.
$$
Now since for fixed $\epsilon>0$ sufficiently small,
\begin{align*}
\int_0^\pi \frac{|\sin{(N+1/2)x|}}{|\sin(x/2)|}\,\ud x
&\le\int_0^{\epsilon/N}(2N+1)\,\ud x+\int_{\epsilon/N}^\epsilon \frac{1}{x/2}\,\ud x
         +\int_\epsilon^\pi \frac{|\sin{(N+1/2)x|}}{|\sin(x/2)|}\,\ud x\\
&\le (2+\frac{1}{N})\epsilon+2\log N+C(\epsilon)\\
&\le c\,\log N\quad\mbox{ as }N\to\infty,
\end{align*}
we have $\|D_N(\xi)\|_{L^1(-\pi,\pi)}=O(\log N)$ and therefore
$$
|u^{N}_F(x,y)-u(x,y)|\le C\,\|u\|_{C^t(D)}\,N^{-t}\,(\log N)^2,
\qquad\mbox{for any }\,x,y\in\mathbb{T}^2.
$$
Similarly in the three-dimensional case one can show that
$$
\|u^{N}_F-u\|_{L^\infty(D)}\le C\,\|u\|_{C^t(D)}\,N^{-t}\,(\log N)^3.
$$
Since $t$ can be chosen arbitrarily close to $s-\frac{d}{2}$ we obtain $\|P^{\bot}\|_{{\cal L}(X,X'')}={\cal O}(N^{-t})$ for any $t<s-\frac{d}{2}$. This, since $X'=X''$, implies that $\|P^{\bot}\|_{{\cal L}(X,X')}$ is bounded independently of $N$.  The result follows by
Theorem \ref{t:tapp}.
\end{proof}

%
%

\section{Conclusion}\label{sec:conclu}
We have addressed the inverse problem of finding the diffusion coefficient
in a uniformly elliptic PDE in divergence form, when noisy observations of its solution are given, using a Bayesian approach.
We have applied the results of \cite{CDRS09} on well-definedness and well-posedness of the posterior measure to show that for an appropriate choice of prior measure 
this inverse problem is well-posed. We also provided a general theorem concerning weak
approximation of the posterior using finite
dimensional truncation of the Karhunen-Lo\`eve expansion:
Theorem \ref{t:tapp}.
We have then used the result of this theorem 
to give an estimate of the weak error in the posterior measure 
when using Fourier truncation: 
Theorem \ref{t:appF}.
Future work arising from the results in this paper
includes the possibility of application to other
inverse problems, the study of rare events and
the effect of approximation, and the question of obtaining
improved rates of weak convergence under stronger
conditions on the mapping $G$.
Also of interest is the extension to non-Gaussian
priors of Besov type \cite{Las09}.



\renewcommand{\theequation}{A-\arabic{equation}}
\setcounter{equation}{0}  

\renewcommand{\thelemma}{A-\arabic{lemma}}

\renewcommand{\thecorollary}{A-\arabic{corollary}}

\renewcommand{\thetheorem}{A-\arabic{theorem}}
\setcounter{theorem}{0}  

\section*{Appendix A}

Let $D\subset \R^d$ be open and bounded and $p\in W_0^{1,q}(D)$ satisfy the following integral identity
\begin{align}\label{eq:weak}
\int_D \nabla v\cdot (a\nabla p+e)+f v\,\ud x=0,
\end{align}
for any $v\in C^1_0(D)$.
We find an estimate for the $W^{1,q}$ norm of $p$ with special attention on how the upper bound depends on the diffusion coefficient $a$. The results of this Appendix are obtained using slight modification of the proof of Shaposhnikov \cite{Sh06} for our purpose here. 

In the following, Lemma \ref{l:locale} gives an estimate for the $W^{1,q}$ norm of
$p$ over $B$, a ball of sufficiently small radius in $\R^d$ and of centre 
$0\in\R^d$. 
Lemma \ref{l:locale2} gives a similar result for the case that 
the domain is $B\cap\{(x^1,\dots,x^d)\in\R^d:x^d\ge0\}$. 
This lemma allows us to consider the effect of the boundary when generalizing
to a bounded domain $D$.
Lemma \ref{l:locale} and \ref{l:locale2} are then used to prove Theorem \ref{t:est}
which gives an estimate for $\|p\|_{W^{1,q}(D)}$ in a general bounded domain $D$.
Finally in Corollary \ref{c:est} we consider $a\in C^t(D)$ and obtain an estimate
for $\|p\|_{W^{1,q}}$ with a polynomial dependence on $\|a\|_{C^t(D)}$.

{\bf Notation:} In this appendix we use the operators $P_r$ and $Q_r$ which for $u$ a vector valued function and $g$ a scalar function are defined as
$$
P_r(g)(y)=\int_{B(0,r)} K(x-y)\,g(x)\,\ud x\,\qquad 
Q_r(u)(y)=\int_{B(0,r)} \nabla K(x-y)\cdot u(x)\,\ud x,
$$
where
$$
K(y)=\left\{
\begin{array}{ll}
\frac{-1}{d(d-1)\alpha(d)|y|^{d-2}},& d>2,\\
\frac{1}{2\pi}\ln|y|,& d=2.
\end{array}
\right.
$$
Also for any $r>0$ we define
$$B(0,r)=\{x\in\R^d:|x|<r\}.$$

\begin{lemma}\label{l:locale}
Let $a\in C(B(0,1))$ and $a(0)=a_0$.
There exists $r<1$ such that if
$p\in W^{1,q}_0(B(0,r))$ satisfies (\ref{eq:weak}) with $D\equiv B(0,r)$,
$P_rf\in W^{1,q}(B(0,r))$, $e\in L^q(B(0,r))$,
and $\mathrm{supp}\ p,\,\mathrm{supp}\ f,\,\mathrm{supp}\ e\in B(0,r)$,
then, letting $B=B(0,r)$,
\begin{equation}\label{eq:locale}
\|p\|_{W^{1,q}(B)}\le \frac{1}{\lambda}\frac{C_0(d,q)}{C_1}(1+r)\,\big(\|e\|_{L^q(B)}+\|P_r(f)\|_{W^{1,q}(B)}\big),
\end{equation}
with 
$$
C_1=1-\frac{(1+r)C_0(d,q)}{\lambda}\sup_{B}|a-a_0|>0.
$$
\end{lemma}
%
\begin{lemma}\label{l:locale2}
Let $G(0,r)=B(0,r)\cap \{(x^1,\cdots,x^d):x^d\ge 0\}$ for $r>0$.
Suppose that in the assumptions of Lemma (\ref{l:locale}), $B(0,1)$ and $B(0,r)$ are replaced with $G(0,1)$ and $G(0,r)$. Assume that the functions $p,f$ and $e$ vanish
in a neighborhood of the spherical part of the hemisphere $G(0,r)$ but not
necessarily on $G(0,r)\cap\{x^d=0\}$. Then the result of Lemma \ref{l:locale} holds.
\end{lemma}

The above results follow from the proofs of Lemma 2 and 3 of \cite{Sh06}.
%
%
\begin{theorem}\label{t:est}
Assume that $D\subset\R^d$, $d=2,3$, is a bounded $C^1$ domain, and  $a\in C(\overline{D})$ with
$$
0<\lambda\le a<\Lambda<\infty.
$$
Suppose also that $e\in L^q(D)$ and $f\in L^2(D)$.
Then for $2<q<6$
$$
\|p\|_{W^{1,q}}\le \frac{C(D,d,q)}{\lambda}\Big(1+\frac{1}{\delta^{1+d}}\Big) (1+\delta)(1+\Lambda)(\|e\|_{L^q(D)}+\|f\|_{L^2(D)}),
$$
where $\delta$ is a positive constant that satisfies the following:
$$
\max_{y\in B(x,\delta)}|a(y)-a(x)|\le \frac{\lambda}{4C_0(d,q)},\quad\mbox{for any } x\in D.
$$
where $C_0(d,q)$ is the constant in Lemma \ref{l:locale}. 
\end{theorem}
\begin{proof} 
Choose $r_D$ so that for any $x\in D$ there exists a ball
of radius $r_D$ inside $D$ that contains $x$.
Let $r=\min\{r_D,\delta\}$. Corresponding to $r$,
consider $\{x_j\}_{j=1}^J\subset\overline D$ so that the set of 
the neighborhoods of these points, $\{U(x_j)\}_{j=1}^J$ defined as follows, forms a cover of 
$\overline D$:
\begin{itemize}
\item[-] for $x_j\in D$, $U(x_j)=B(x_j,r)$, 

\item[-] for $x_j\in\partial D$, $U(x_j)=B(x_j,r)\cap\overline{D}$ 
and there exists $C^1$ mapping $\psi_j$ such that
$\psi_j(U(x_j))=G(x_j,r)$ where $G(x_j,r)=B(x_j,r)\cap \{(x^1,\cdots,x^d):x^d\ge x_j^d\}$,

\item[-] and there exists a partition of unity $\{\xi_j\}_{j=1}^J$ subordinate to $\{U(x_j)\}$.
\end{itemize}
Let $w_j=\xi_j p$. Hence $p=\sum_{j=1}^J w_j$. Define
$$
\tilde{e}=\xi_j e-a\,p\,\nabla\xi_j 
\quad\mbox{and}
\quad\tilde{f}=\xi_j f+a\nabla \xi_j\cdot\nabla p.
$$
On $U(x_j)$ with $x_j\in D$, $w_j$ satisfies
$$
\int_{B(x_j,r)}\nabla v\cdot (a\nabla w_j+\tilde e)+v\tilde f\,\ud x=0.
$$
By Lemma \ref{l:locale}, with $B_j=B(x_j,r)$, we have
$$
\|w_j\|_{W^{1,q}(B_j)}\le \frac{C(d,q)}{\lambda} (1+r) \big(\|\tilde e\|_{L^q(B_j)}+\|P_r(\tilde f)\|_{W^{1,q}(B_j)}\big).
$$
To estimate $\|\tilde e\|_{L^q(B_j)}$ we write
\begin{align*}
\|\tilde e\|_{L^q(B_j)}
&\le \|\xi_j\,e\|_{L^q(B_j)}+\|a\nabla \xi_j\, p\|_{L^q(B_j)}\\
&\le (1+\Lambda)C(\xi_j,\nabla \xi_j)\,(\|e\|_{L^q}+\|p\|_{L^q)}\\
&\le \frac{c(1+\Lambda)}{r}\,(\|e\|_{L^q}+\|p\|_{L^q}).
\end{align*}
For $\|P_r(\tilde f)\|_{W^{1,q}(B_j)}$, by Sobolev embedding theorem (assuming that $2\le q\le 6$) and since by Theorem 9.9 of \cite{GT83} $\|P_r g\|_{H^2}\le C(d)\|g\|_{L^2}$, we have
\begin{align*}
\|P_r(\tilde f)\|_{W^{1,q}(D)}
&\le\, C(D)\|P_r(\tilde f)\|_{H^2}\\
&\le\, C(D,d)\|\tilde f \|_{L^2}\\
&\le\, C(D,d)(\|\xi_j f\|_{L^2}+\|a\nabla \xi_j\cdot\nabla p\|_{L^2})\\
&\le\, C(D, d,\xi_j,\nabla\xi_j)(1+\Lambda)(\|f\|_{L^2}+\|\nabla p\|_{L^2})\\
&\le\, \frac{C(D, d)(1+\Lambda)}{r}(\|f\|_{L^2}+\|\nabla p\|_{L^2}).
\end{align*}
Since $2<q<6$ and $D$ is bounded, 
$$
\|p\|_{L^q}\le c\,\|\nabla p\|_{L^2}\le c\,(\|f\|_{L^2}+\|e\|_{L^2})
$$
where the second inequality is obtained by taking the inner product of 
(\ref{eq:weak}) with $p$ (noting that $p\in W_0^{1,q}(D)$ and hence $p|_{\partial D}=0$ in the trace sense). Therefore
\begin{align}\label{eq:wj}
\|w_j\|_{W^{1,q}(B_j)}
\le \frac{C(D,d,q)}{r\,\lambda}(1+\Lambda)(1+r)(\|e\|_{L^q}+\|f\|_{L^2}).
\end{align}
It remains to consider the case that $x_j\in \partial D$. For such $x_j$, 
$w_j$ on $U(x_j)$ satisfies (using the map $\psi_j$ defined at the beginning of the proof)
$$
\int_{G(x_j,r)}\nabla v\cdot (\hat a\nabla w_j+\hat e)+v\hat f\,\ud x=0.
$$
where $\hat{a},\hat{e}$ and $\hat{f}$ depend on $\nabla\psi_j$. It is not difficult to see that $|\nabla \psi_j|<C$ where $C$ only depends on the properties of the boundary of $D$, therefore in a similar way to the above argument and using Lemma \ref{l:locale2} it can be shown that
\begin{align*}
\|w_j\|_{W^{1,q}(G_j)}
\le \frac{C(D,d,q)}{r\,\lambda}(1+\Lambda)(1+r)(\|e\|_{L^q}+\|f\|_{L^2}).
\end{align*}
Now we can write
\begin{align*}
\|p\|_{W^{1,q}}
&\le \sum_{j=1}^J\|w_j\|_{W^{1,q}}\le \frac{c\,|D|}{r^d}\;\|w_j\|_{W^{1,q}}\\
&\le \frac{C(D,d,q)}{r^{1+d}\,\lambda}(1+\Lambda)(1+r)(\|e\|_{L^q}+\|f\|_{L^2})\\
&\le \frac{C(D,d,q)}{\lambda}\Big( 1+\frac{1}{\delta^{1+d}} \Big)\,\Big( 1+\frac{1}{r_D^{1+d}} \Big)
           \,(1+\Lambda)(1+r_D)(1+\delta)(\|e\|_{L^q}+\|f\|_{L^2})\\
&\le \frac{C(D,d,q)}{\lambda}\Big( 1+\frac{1}{\delta^{1+d}} \Big)
           \,(1+\Lambda)(1+\delta)(\|e\|_{L^q}+\|f\|_{L^2})
\end{align*}
and the result follows.
\end{proof}


In order to quantify $\delta$ in the above theorem, 
in terms of the norm of the space that $a$ lives in, 
we need to assume $a$ to be H\"older continuous:

\begin{corollary}\label{c:est}
Suppose that the assumptions of Theorem \ref{t:est} holds.
Assume also that $a$ is $t$-H\"older continuous in $D$.
Then 
$$
\|p\|_{W^{1,q}}\le \frac{C(D,d,q,t)}{\lambda}\Big(1+\frac{\|a\|_{C^t(D)}^{(1+d)/t}}{\lambda^{(1+d)/t}}\Big)
\Big(1+\frac{\|a\|_{C^t(D)}^{1/t}}{\lambda^{1/t}}\Big)(1+\Lambda)(\|e\|_{L^q(D)}+\|f\|_{L^2(D)}).
$$
\end{corollary}\vskip.5cm

\begin{proof}
Since $|a(x)-a(y)|\le \|a\|_{C^t}|x-y|^t$, $\delta$ of Theorem \ref{t:est} satisfies
$\delta\le c\lambda^{1/t}\|a\|_{C^t(D)}^{-1/t}$
and the result follows. 
\end{proof}


\section*{Acknowledgements}

The authors are grateful to Jose Rodrigo and Christoph Schwab for helpful discussions.
AMS is grateful to the EPSRC (UK) and ERC for financial support.

\bibliographystyle{99}

\begin{thebibliography}{1} 
{
\bibitem{Ad75} Adams R. A. 1975, {\it Sobolev spaces.} Pure and Applied Mathematics, Vol. 65. Academic Press [A subsidiary of Harcourt Brace Jovanovich, Publishers], New York-London.

\bibitem{BNT07} Babu\v{s}ka I., Nobile F. and Tempone R. 2007, A stochastic collocation method for elliptic partial differential equations with random input data. {\it SIAM J. Numer. Anal.}  {\bf 45} , no. 3, 1005--1034 (electronic). 

\bibitem{BTZ04} Babu\v{s}ka I., Tempone R. and Zouraris G. E. 2004, Galerkin finite element approximations of stochastic elliptic partial differential equations. {\it SIAM J. Numer. Anal.} {\bf 42} , no. 2, 800--825. 

\bibitem{BSw09} Bieri M. and Schwab C. 2009, Sparse high order FEM for elliptic sPDEs. {\it Comput. Methods Appl. Mech. Engrg.}  {\bf 198}, no. 13-14, 1149--1170. 

\bibitem{Bog98} Bogachev V. I. 1998, {\it Gaussian measures.} Mathematical Surveys and Monographs, 62. American Mathematical Society, Providence, RI.

\bibitem{Cha10} Charrier, J. 2010, {Strong and weak error estimates for elliptic partial differential equations with random coefficients}, {\it Submitted}.

\bibitem{CDS10} Cohen, A., DeVore, R. and Schwab, C. 2010,
{\it Analytic regularity and polynomial approximation of parametric
and stochastic elliptic PDEs}.

\bibitem{CDRS09} Cotter S. L., Dashti M., Robinson J.C. and Stuart A.M. 2009, 
Bayesian inverse problems for functions with applications in fluid mechanics. {\it Inverse Problems}  25, 115008.

\bibitem{CDaS10} Cotter S. L., Dashti M., Stuart A.M. 2010, Approximation of Bayesian inverse problems in
differential equations. {\it SIAM J. Numer. Anal.} 48, No. 1, 322-345.


\bibitem{CRSW10} Cotter S. L., Roberts, G.O., 
Stuart A.M. and White, D. 2010, Submitted. 


\bibitem{DaZa92} Da Prato G. \& Zabczyk J. 1992, {\it Stochastic equations in infinite dimensions}. Encyclopedia of Mathematics and its Applications, 44. Cambridge University Press, Cambridge.

\bibitem{Dud02} Dudley, R. M. 2002. {\it Real Analysis and Probability}, 2nd ed.
Cambridge Univ. Press.


\bibitem{FSwT05} Frauenfelder P., Schwab C. and Todor R. A. 2005, Finite elements for elliptic problems with stochastic coefficients. {\it Comput. Methods Appl. Mech. Engrg.}  {\bf 194}, no. 2-5, 205--228.

\bibitem{GT83} Gilbarg D. \& Trudinger N.S. 1983, {\it Elliptic partial differential equations of second order}. Springer-Verlag, Berlin.

\bibitem{HaiS07} Hairer, M., Stuart, A. M. and Voss J. 2007, Analysis of SPDEs arising in path sampling. II. The nonlinear case.  {\it Ann. Appl. Probab.}  {\bf 17},  no. 5-6, 1657Ð1706. 

\bibitem{Las09} Lassas M., Saksman E. and Siltanen S. 2009, Discretization invariant Bayesian inversion and Besov space priors. 
{\it Inverse Problems and Imaging}  {\bf 3}, 87-122.

\bibitem{Lif95} Lifshits, M. A. 1995,
{\it Gaussian random functions. }
Mathematics and its Applications, 322. Kluwer Academic Publishers, Dordrecht.

\bibitem{Mat08} Matthies H. G. 2008, Stochastic finite elements: computational approaches to stochastic 
partial differential equations. {\it ZAMM Z. Angew. Math. Mech.}, {\bf 88} (11):849Ð873.

\bibitem{MatB99} Matthies H. G. and Bucher C. 1999, Finite elements for stochastic media problems. 
{\it Comput. Methods Appl. Mech. Engrg.}, {\bf 168} (1-4): 3Ð17. 


\bibitem{McT96} McLaughlin D. and Townley L. (1996), A reassessment of the ground water inverse problem,{\it Water Resour. Res.} {\bf 32},1131Ð1161. 

\bibitem{NTW08a} Nobile F., Tempone R. and Webster C. G. 2008, An anisotropic sparse grid stochastic collocation 
method for partial differential equations with random input data. 
{\it SIAM J. Numer. Anal.}, 
{\bf 46} (5):2411Ð2442. 

\bibitem{NTW08b} Nobile F., Tempone R. and Webster C. G. 2008, A sparse grid stochastic collocation method for 
partial differential equations with random input data. {\it SIAM J. Numer. Anal.}, 
{\bf 46}(5): 2309Ð2345. 



\bibitem{ST06} Schwab C. \& Todor R. A. 2006, Karhunen-Lo\`eve approximation of random fields by generalized fast multipole methods. {\it J. Comput. Phys.} {\bf 217},  no. 1, 100--122.

\bibitem{Sh06} Shaposhnikov, S. V. 2006, {\it On Morrey's estimate for the Sobolev norms of solutions of elliptic equations}. (Russian)  Mat. Zametki  79  (2006),  no. 3, 450--469;  translation in  Math. Notes  79  no. 3-4, 413--430.

\bibitem{St10} Stuart A. M., 2010, Inverse Problems: A Bayesian Approach. 
{\it Acta Numerica}. 

\bibitem{Zyg88} Zygmund, A. 1988. {\it Trigonometric series.} Vol. I, II. Reprint of the 1979 edition. Cambridge Mathematical Library. Cambridge University Press, Cambridge.

\bibitem{ZhL04} Zhang D. and Lu Z. 2004, An efficient, high-order perturbation approach for flow in 
random porous media via karhunen-lo`eve and polynomial expansions. 
{\it J. Comput. Phys.},
{\bf 194} (2): 773Ð794.

}


\end{thebibliography}

\end{document}